\documentclass[reqno]{amsart}
\usepackage{amsmath, amssymb, latexsym, amsthm, amsfonts, mathrsfs,
  amscd,color, bbm}
\newcommand\reallywidehat[1]{%
\savestack{\tmpbox}{\stretchto{%
  \scaleto{%
    \scalerel*[\widthof{\ensuremath{#1}}]{\kern-.6pt\bigwedge\kern-.6pt}%
    {\rule[-\textheight/2]{1ex}{\textheight}}
  }{\textheight}%
}{0.5ex}}%
\stackon[1pt]{#1}{\tmpbox}%
}
\renewcommand{\eqref}[1]{(\ref{#1})}   

\numberwithin{equation}{section}

\theoremstyle{plain}
\newtheorem{theorem}{Theorem}[section]
\newtheorem{lemma}[theorem]{Lemma}
\newtheorem{corollary}[theorem]{Corollary}
\newtheorem{proposition}[theorem]{Proposition}

\theoremstyle{definition}

\newtheorem{remark}[theorem]{Remark}

\newtheorem{example}[theorem]{Example}

\theoremstyle{definition}

\newcommand{\R}{{\mathbbm R}}

\newcommand{\C}{{\mathbbm C}}
\newcommand{\N}{{\mathbbm N}}

\newcommand{\PP}{{\mathbbm P}}

\begin{document}
\title[Connected components of Berkovich fixed locus]{Connected components of Berkovich fixed locus: Potential good reduction}
\date{August 11, 2025}
\author{Niladri Patra}
\address{Reasearch Associate - 1, \\ 
Indian Statistical Institute, Delhi Centre, \\
S. J. S. Sansanwal Marg, New Delhi, India.}
\email{niladri.patra@isid.ac.in, niladript@gmail.com}


\begin{abstract}
    Let $\PP^{1,an}$ be the Berkovich projective line over a complete, algebraically closed, non-Archimedean field. Let $\phi$ be a degree $\geq 2$ rational map with potential good reduction, acting on $\PP^{1,an}$. In this article, we study the topology of the fixed locus of $\phi$. we show that the reduction of $\phi$ at its type~II totally ramified fixed point dictates the topological structure of the fixed locus of $\phi$. We give an easily verifiable equivalent criterion for the fixed locus of $\phi$ to be connected as well as an equivalent criterion for the fixed locus of $\phi$ to be finite. Moreover, we provide a sharp upper bound for the number of connected components of the fixed locus of a rational map with potential good reduction. 
\end{abstract}

\maketitle

\section{Introduction}

Let $K$ be a complete, algebraically closed, non-Archimedean field, of arbitrary characteristic. We assume the non-Archimedean absolute value $|\cdot |$ of $K$ to be non-trivial. Non-Archimedean dynamics on $K$ studies behaviour of self-iterates of rational maps, defined over $K$, acting on the projective line over $K$. It is nowadays standard to extend this classical version of non-Archimedean dynamics to the dynamics on the Berkovich projective line, denoted $\PP^{1,an}$ (As the field $K$ underneath will remain constant throughout this article, we have suppressed it from the notation of the Berkovich projective line). For detailed discussion on Berkovich projective line and the action of rational maps on it, see \cite{berk}, \cite{brume}, \cite{benebook}, for a short introduction, see \cite{faber1}. In summary, Berkovich projective line over $K$ is a path connected, compact, Hausdorff space under the weak topology in which the classical projective line $\PP^{1}(K)$ embeds topologically as a dense subset (see \cite{brume}, \cite{benebook}). The action of any rational map on $\PP^{1}(K)$ extends continuously and uniquely to an action on the Berkovich projective line. As one studies action of self-iterates of rational maps on the Berkovich projective line, the simplest type of points one encounters are fixed points. To elaborate, a \emph{fixed point} of a rational map $\phi$ defined over $K$ is a point $\xi \in \PP^{1,an}$, such that $\phi(\xi)=\xi$. Any rational map $\phi$ defined over $K$ must have at least one fixed point in $\PP^{1}(K)$. Over a field, a rational map has only finitely many fixed points. However, on the Berkovich projective line the set of fixed points of a rational map can be infinite. Thus, we define \emph{the fixed locus} of a rational map $\phi$ to be the set of all fixed points of $\phi$ in $\PP^{1,an}$. The set of all fixed points of a continuous self-map on a Hausdorff space is closed. Thus, for any rational map $\phi$, the fixed locus of $\phi$ in $\PP^{1,an}$ is non-empty, closed and compact under the weak topology.

Consider the complete, algebraically closed, non-Archimedean field $K$ equipped with the non-trivial non-Archimedean absolute value $|\cdot|$. The set of all elements of $K$ with absolute value less than or equal to $1$ forms a local ring, denoted $\mathcal{O}_{K}$. The unique maximal ideal of $\mathcal{O}_{K}$ comprises of all elements of $\mathcal{O}_{K}$ with absolute value less than $1$, denoted $m_{K}$. The field $\mathcal{O}_{K}/m_{K}$ is called the residue field of $K$, denoted $\tilde{K}$. Consider a rational map $\phi$ defined over $K$. By multiplying the numerator and denominator of $\phi$ by a suitable rational function, we can assume that the numerator and denominator of $\phi$ has no non-constant common factor and all the coefficients of $\phi$ lie in $\mathcal{O}_{K}$. We will call this \emph{a normalized form} of the rational map $\phi$. Reducing the coefficients of $\phi$ modulo the maximal ideal $m_{K}$, one obtains a rational map $\tilde{\phi}$ defined over $\tilde{K}$ (note that we allow the map $\tilde{\phi}$ to be a constant, even the constant map $\infty$). This map $\tilde{\phi}$ is called the \emph{reduction} of $\phi$ or \emph{the residual map} of $\phi$. A normalized map $\phi$ is said to have a \emph{good reduction}, if degree of $\tilde{\phi}$ is the same as the degree of $\phi$. The map $\phi$ is said to be of \emph{potential good reduction} (in some of the literature, it is also referred as relative good reduction) if there exists a M\"{o}bius transformation $\sigma \in PGL_{2}(K)$ such that a normalized version of the conjugation $\sigma \circ \phi \circ \sigma^{-1}$ has good reduction. 

Let $\phi(z)$ be a rational map defined over an algebraically closed field $K$. Let $\alpha$ be a point in $\PP^{1}(K)$, the classical projective line over $K$. One can always find a M\"{o}bius transformation $\sigma \in PGL_{2}(K)$ such that both $\sigma(\alpha)$ and $\sigma(\phi(\alpha))$ lie in $K$. The point $\alpha$ is called a \emph{critical point} of $\phi$ if $\sigma(\alpha)$ is a root of the formal derivative $(\sigma \circ \phi \circ \sigma^{-1})'(z)$ in $K$. 

Properties of fixed points of a rational map on the Berkovich projective line have been studied by Rivera-Letelier (\cite{rl1}, \cite{rl2}, \cite{rl3}, \cite{rl4}), Rumely (\cite{rume1}, \cite{rume2}), Benedetto (\cite{bene1}, \cite{bene2}) among others. A rational map with potential good reduction has a unique type~II totally ramified fixed point (see Theorem \ref{tool4}, Lemma \ref{3.5}). The central purpose of this article is to show that the topological structure of the fixed lcous of a rational map with potential good reduction is largely governed by the local behaviour of the map around the totally ramified type~II fixed point. In particular, we discuss connectedness properties of the fixed locus of a rational map with potential good reduction. Apart from connectedness properties, the results in this article serve two purposes. Firstly, they provide more insight into the dynamics of rational maps acting on the Berkovich projective line. Secondly, they subtly highlight an interplay between the ramification locus and the fixed locus of a rational map. We now state the main results of this article. Firstly, we determine the number of connected components of the fixed locus (Theorem \ref{maintheorem}).

\begin{theorem}\label{1.1}
    Let $\phi$ be a rational map of degree $\geq 2$, defined over $K$, with potential good reduction. The number of connected components of the fixed locus of $\phi$ is $1 + $ the number of attracting fixed points of $\phi$.
\end{theorem}

This provides us with a bound of $(d+2)$ for the number of connected components of the fixed locus of a rational map of degree $d \geq 2$. Using Theorem \ref{1.1}, we also give a new proof of the main theorem of \cite{bene1} in Remark \ref{rem:bene_new_proof}.

Next we give equivalent criterions for the extreme cases of Theorem \ref{1.1} in Corollaries \ref{3.2} and \ref{3.11}.

\begin{corollary}\label{cor:intro_conn}
    Let $\phi$ be a degree $\geq 2$ rational map defined over $K$, with potential good reduction. Then the fixed locus of $\phi$ is connected if and only if the reduction of $\phi$ at the type~II totally ramified fixed point has no fixed critical point. 
\end{corollary}

\begin{corollary}\label{cor:intro_fin}
    Let $\phi$ be a degree $d \geq 2$ rational map with potentially good reduction. Then, The fixed locus of $\phi$ is a finite set if and only if every fixed point of the reduction of $\phi$ at the type~II totally ramified fixed point is critical. In such a case, the fixed locus of $\phi$ consists of the classical fixed points of $\phi$ and the type~II totally ramified fixed point of $\phi$.
\end{corollary}

We provide examples (Example \ref{3.12}), one for each degree, that shows the bound observed from Theorem \ref{1.1} is sharp. This class of examples also serves Corollary \ref{cor:intro_fin}. However, each of these examples have inseparable reduction at the Gauss point, which is the type~II totally ramified fixed point. In Example \ref{ex:sep_fin}, we provide a rational map which has separable reduction at the type~II totally ramified fixed point, with finitely many fixed points in $\PP^{1,an}$.   

Next, in Corollary \ref{3.6} and Proposition \ref{noid}, we characterize the entire fixed locus of a rational map $\phi$ with potential good reduction. The set $A_{\phi}$ is defined to be the set of all indifferent classical fixed points of $\phi$ and the totally ramified type~II fixed point in $\PP^{1,an}$. The set Hull$(A_{\phi})$ denotes the connected hull of $A_{\phi}$ i.e. the smallest connected subset of $\PP^{1.an}$ containing $A_{\phi}$.

\begin{proposition}\label{prop:conn_hull}
    Let $\phi$ be a rational map of degree $\geq 2$ with potential good reduction. Then, Hull$(A_{\phi})$ is contained in the fixed locus of $\phi$. If there is no point in Hull$(A_{\phi})$ at which the reduction of $\phi$ is the identity map, then the fixed locus of $\phi$ consists of Hull$(A_{\phi})$ and the attracting fixed points of $\phi$. 
\end{proposition}

Using Proposition \ref{prop:conn_hull}, we determine the entire fixed locus of a particular rational map in Example \ref{3.7}. 

We then extend the results above to the set of all periodic point of a rational map with potential good reduction.

In the next section, we introduce briefly the necessary definitions, notations and properties needed for this article. We also provide a collection of results with references, that will be required in the later section. In third section, we prove the main results of this article.

\section*{Acknowledgement}

The author would like to thank C. S. Rajan, Sabyasachi Mukherjee, Xander Faber, Juan Rivera-Letelier, Shubham Kumar for many helpful comments and discussions. The author would also like to thank Shrestha Pattanayak for her support.

\section{Prerequisites}

Our notations will roughly follow the notations of the book \cite{brume} by Baker and Rumely. For the construction and basic properties of Berkovich projective line $\PP^{1,an}$, see \cite{berk}, \cite{brume}, \cite{benebook}. The points of $\PP^{1,an}$ are classified as follows,

$(1)$ Type~I points (also called, \emph{rigid} or \emph{classical} points), which correspond to (evaluation semi-norm at) points of $\PP^{1}(K)$.

$(2)$ Type~II points which correspond to (sup norm on) closed disks in $K$ with radius lying in the value group, $|K^{\times}|$.

$(3)$ Type~III points which correspond to (sup norm on) closed disks in $K$ with radius not lying in the value group.

$(4)$ Type~IV points which correspond to (limit of sup norms on) decreasing sequences of closed disks whose intersection is empty.

The type~II point corresponding to the closed disk $D(0,1)$ of radius $1$ around $0$ is called the \emph{Gauss point}, denoted $\zeta_{Gauss}$. Any other type~II or III point that corresponds to a closed disk of radius $r$ around $a \in K$, is denoted as $\zeta_{a,r}$. The point $\zeta_{a,r}$ is said to have \emph{diameter} (Compare \cite[Page~11]{brume}) diam$(\zeta_{a,r})$ $= r$, as it is the supremum of distances between any pair of points on $D(a,r)$. The \emph{diameter} of a type~IV point is defined to be the limit of the diameters of a decreasing sequence of disks defining that type~IV point. The Berkovich projective line can be seen as an infinite tree, rooted at $\zeta_{Gauss}$ (See Chapter $2$, \cite{brume}). This tree branches at every type~II point in infinitely many \emph{directions}, which are in one-one correspondence with $\PP^{1}(\tilde{K})$. At any type~III point, it branches in only two directions. Type~I and IV points serve as endpoints of this tree (See picture at \cite[page~30]{brume}).    

\subsection{Topology on the Berkovich projective line}

Consider the connected (connected as a tree) components of the Berkovich projective line missing a point. The \emph{weak topology} (or, \emph{Berkovich topology}) (See \cite[Page~28]{brume}) is given on the Berkovich projective line by considering the set of all such connected components as a subbasis. Under the weak topology, the Berkovich projective line is a uniquely path-connected (i.e. any non-trivial loop in $\PP^{1,an}$ must involve retracing), compact, Hausdorff space (See \cite[Proposition~2.6~and~Lemma~2.10]{brume}). This topology restricts to the usual non-Archimedean topology on the type~I points. Moreover, under this topology the type~I points form a dense subset (See \cite[Lemma~2.9]{brume}). Let $\zeta$ be a point on the Berkovich projective line. Consider the connected components of $\PP^{1,an} \backslash \zeta$. These are called \emph{open Berkovich disks} (Compare \cite[Definition~2.25]{brume}) at $\zeta$. These are also considered as \emph{directions} at $\zeta$ (Compare \cite[Page~41]{brume}). Collection of all these directions at $\zeta$ forms the \emph{tangent space} (Compare \cite[Page~402]{brume}) at $\zeta$, denoted $T_{\zeta} \PP^{1,an}$. We denote a direction at a point on the Berkovich projective line using a vector, such as $\vec{v}$. If $\vec{v}$ is a direction at the point $\zeta \in \PP^{1,an}$, then the open Berkovich disk corresponding to it is denoted as $B_{\vec{v}}$. Let $\xi$ be another point on $\PP^{1,an}$ distinct from $\zeta$. Then the connected component (open Berkovich disk) of $\PP^{1,an} \backslash \zeta$ containing $\xi$ will be denoted as $B_{\xi}(\zeta)$. 

\subsection{Arcs and the big metric} 

Let $\zeta$ and $\xi$ be two points on $\PP^{1,an}$. The notation $[\zeta,\xi]$ (Compare \cite[Page~9]{brume}) stands for the unique closed path in $\PP^{1,an}$ starting at $\zeta$, ending at $\xi$, with no retracing. This is also called the \emph{arc} between $\zeta$ and $\xi$. The notations $[\zeta, \xi), (\zeta, \xi], (\zeta,\xi)$ correspond to the paths between $\zeta$ and $\xi$ with one or both endpoints removed, as designated by round brackets. 

The space $\PP^{1,an} \backslash \PP^{1}(K)$ is called the \emph{Berkovich hyperbolic space} (See \cite[definition~2.20]{brume}). On this space, one defines a metric called \emph{the big metric} (See \cite[Page~44~,~Equation~(2.23)]{brume}). Let $x$ and $y$ be two points in the Berkovich hyperbolic space. Denote the unique point lying on the intersection of the three arcs, $[x, \infty]$, $[y, \infty]$ and $[x,y]$ as $x \vee_{\infty} y$. Then \emph{the big metric} $\rho$ is defined to be,

\begin{equation}\label{bigmet}
    \rho(x,y) = 2 \log_{v}(\text{diam}(x \vee_{\infty} y)) - \log_{v}(\text{diam}(x)) - \log_{v}(\text{diam}(y)),
\end{equation}
where $v$ denotes the absolute value of the element of $K$ which normalizes $|\cdot|$ (one can choose this element to be a uniformizer of a discrete valuation ring topologically and algebraically embedded in $\mathcal{O}_{K}$). Under the big metric $\rho$, the type~I points lie at an infinite distance from any other type points. The topology induced by the big metric on the Berkovich hyperbolic space is called the \emph{strong topology} (See \cite[Section~2.7]{brume}). As the name suggests, this is a stronger topology than the restriction of the weak topology on the hyperbolic space. However, when one considers the Berkovich projective line as a projective limit of finite $\R$ trees, the projective limit topology generated by the restriction of strong topology on the finite subtrees of the Berkovich projective line, is the same as the Berkovich topology on $\PP^{1,an}$ (See \cite[Theorem~2.21]{brume}).

\subsection{Action of rational maps and multiplicity}

Action of any rational map, defined over $K$, on $\PP^{1}(K)$ extends continuously and uniquely to an action on the Berkovich projective line (See \cite[Section~2.3]{brume}). Under these actions, type of a point stay preserved. Let $\phi$ be a rational map that sends $\zeta \in \PP^{1,an}$ to $\phi(\zeta)$. Then, $\phi$ induces a map at level of tangent spaces, $T_{\zeta}\phi : T_{\zeta} \PP^{1,an} \rightarrow T_{\phi(\zeta)} \PP^{1,an}$. If $\zeta$ is a type~II fixed point of $\phi$, the map $T_{\zeta}\phi$ is called the \emph{reduction of} $\phi$ \emph{at} $\zeta$ or, the \emph{residual map of} $\phi$ \emph{at} $\zeta$.

There is a notion of multiplicity (See \cite[Definition~9.7]{brume}, \cite[Section~3.1]{faber1}) of a rational map $\phi$ at a point on the Berkovich projective line, which extends the usual notion of multiplicity of a function on $\PP^{1}(K)$. Let $\phi$ be a rational map of degree $d$. The \emph{multiplicity} (also called \emph{local degree} or \emph{ramification index}) of $\phi$ at a point $\zeta$ is denoted $m_{\phi}(\zeta)$. For every direction $\vec{v}$, there is also a notion of \emph{directional multiplicity} (See \cite[Page~265]{brume}, \cite[Section~3.2]{faber1}) of $\phi$ at $\zeta$ in the direction $\vec{v}$, denoted $m_{\phi}(\zeta, \vec{v})$. For every direction $\vec{w}$ at $\phi(\zeta)$, the following equality holds (See \cite[Theorem~9.22~,~Part~C]{brume}),

\begin{equation}\label{2.2}
	\sum_{\vec{v} \in (T_{\zeta}\phi)^{-1}(\vec{w})} m_{\phi}(\zeta, \vec{v}) = m_{\phi}(\zeta).
\end{equation} 

Moreover, as one would expect for a degree $d$ rational map, the following equality holds for arbitrary $\zeta \in \PP^{1,an}$ (See \cite[Theorem~9.8~,~Part~C]{brume}),

\begin{equation}\label{2.3}
	\sum_{\xi \in \phi^{-1}(\zeta)} m_{\phi}(\xi) = d.
\end{equation} 

Fix a degree $\geq 2$ rational map $\phi$ defined over $K$. Any point $\zeta \in \PP^{1,an}$ at which the multiplicity of $\phi$ is greater than $1$ is called a \emph{ramified point}. We reserve the notion of \emph{critical point} for type~I ramified points. The set of ramified points of $\phi$ in $\PP^{1,an}$ is called \emph{the ramification locus} of $\phi$ (See \cite[Definition~3.2]{faber1}). As the multiplicity function is upper semicontinuous with respect to the weak topology (See \cite[Proposition~9.28]{brume}), the ramification locus is a closed set. A point in $\PP^{1,an}$ is called \emph{totally ramified} if it has multiplicity equal to the degree of the rational map in a normalized form. Similarly, one defines the \emph{locus of total ramification} (See \cite[Definition~8.1]{faber1}). 

\subsection{Classification of fixed points}

Following Rivera-Letelier, we introduce the following classification of fixed points of a rational map $\phi$, which are \emph{not} of type~I (also called rigid or classical). A non-rigid fixed point of $\phi$ is called \emph{repelling} if the multiplicity of $\phi$ at that point is $2$ or higher. Otherwise, it is called an \emph{indifferent} fixed point.

On classical/rigid/type~I points, we have the usual classification of fixed points. The \emph{multiplier} $\lambda$ of $\phi$ at a fixed point $a \in K$ is defined to be $\lambda = \phi'(a)$, where $\phi'$ denote the derivative of $\phi$. The notion of multiplier is extended naturally to $\infty$ as $\infty$ can be moved to any other rigid point via a M\"{o}bius transformation. The rigid fixed points are classified as,

\begin{itemize}
    \item If $|\lambda|=0$, it is called \emph{superattracting} fixed point,
    \item If $0<|\lambda|<1$, it is called \emph{attracting} fixed point,
    \item If $|\lambda|=1$, it is called \emph{indifferent} fixed point,
    \item If $|\lambda|>1$, it is called \emph{repelling} fixed point.
\end{itemize}

\subsection{A collection of results}

With the notations and background set, we recall a few well-known results that will be needed in the next section. We omit most of the proofs in this section. The following theorem is due to Rivera-Letelier (\cite[Lemma~2.1]{rl1}, \cite[Proposition~9.41]{brume}).

\begin{theorem}\label{tool1}
    Let $\phi$ be a non-constant rational map defined over $K$. Let $\zeta$ be a point in $\PP^{1,an}$. Let $\vec{v}$ be a direction at $\zeta$ and $B_{\vec{v}}$ is the corresponding open Berkovich disk. Suppose the map induced by $\phi$ at the level of tangent space at $\zeta$, $T_{\zeta}\phi : T_{\zeta}\PP^{1,an} \rightarrow T_{\phi(\zeta)}\PP^{1,an}$ sends $\vec{v}$ to $\vec{w}$. Let $B_{\vec{w}}$ be the open Berkovich disk in the direction $\vec{w}$ at $\phi(\zeta)$. Then, either $\phi(B_{\vec{v}}) = B_{\vec{w}}$ or $\phi(B_{\vec{v}}) = \PP^{1,an}$.
\end{theorem}

In the above theorem, if $\phi(B_{\vec{v}}) = B_{\vec{w}}$, then $\vec{v}$ is called a \emph{good direction} at $\zeta$ (compare \cite[Page~1450]{kiwi1}). Otherwise, it is called a \emph{bad direction}.

The following result is due to Rivera-Letelier (\cite[Proposition~3.3]{rl1}, \cite[Corollary~9.25]{brume}).

\begin{theorem}\label{tool2}
    Let $\phi$ be a non-constant rational map, in normalized form, defined over $K$. Then $\phi$ has non-constant reduction if and only if $\zeta_{Gauss}$ is a fixed point of $\phi$. Identifying the tangent space at the Gauss point with $\PP^{1}(\tilde{K})$, we get,

    $(1)$ $m_{\phi}(\zeta_{Gauss}) = deg(\tilde{\phi})$, where $\tilde{\phi}$ is the reduction of $\phi$ with common factors of the denominator and numerator cancelled out.

    $(2)$ $T_{\zeta_{Gauss}}\phi (\vec{v}) = \tilde{\phi}(\vec{v})$, for every $\vec{v}$ in the tangent space of $\zeta_{Gauss}$ which is canonically identified with $\PP^{1}(\tilde{K})$.
    
    $(3)$ $m_{\phi} (\zeta_{Gauss}, \vec{v}) = m_{\vec{v}}(\tilde{\phi})$, for every $\vec{v}$ in $\PP^{1}(\tilde{K})$. Here, $m_{\vec{v}}(\tilde{\phi})$ denotes the multiplicity of $\tilde{\phi}$ at $\vec{v}$.
\end{theorem}

The next theorem is due to Rivera-Letelier (\cite[Proposition~4.6]{rl2}) for $K=\C_{p}$, the field of $p$-adic complex numbers, and generalised by Baker and Rumely (\cite[Corollary~9.21]{brume}) for arbitrary $K$. 

\begin{theorem}\label{tool3}
    Let $\phi$ be a non-constant rational function defined over $K$. Let $a \in \PP^{1,an}$ be a point not of type~I. Let $\vec{v}$ be a tangent direction at $x$. Then there is an arc $[a,c']$ representing the tangent direction $\vec{v}$ such that $\phi([a,c']) = [\phi(a), \phi(c')]$ and for all $x,y \in [a,c']$, the following holds,
    $$\rho(\phi(x), \phi(y)) = m_{\phi}(a, \vec{v}) \rho(x,y).$$
\end{theorem}

The following well-known result is difficult to find in one place. So we give a partial proof of it here.

\begin{theorem}\label{tool4}
    Let $\phi$ be a degree $\geq 2$ rational map, in normalized form, defined over $K$. The following are equivalent,
    
    $(1)$ The map $\phi$ has good reduction.
    
    $(2)$ The Gauss point $\zeta_{Gauss}$ is totally ramified and fixed under $\phi$.

    $(3)$ The Gauss point $\zeta_{Gauss}$ is fixed and every direction at $\zeta_{Gauss}$ is a good direction.

    $(4)$ The Gauss point point is fixed and the multiplicity function is non-increasing on every arc of the form $[\zeta_{Gauss}, \xi]$, $\xi \in \PP^{1,an}$.
\end{theorem}

\begin{proof}
    $(1) \iff (2)$ This is due to Rivera-Letelier (\cite[Theorem~4]{rl4}, \cite[Proposition~10.105]{brume}).

    $(2) \implies (3)$ The Gauss point is totally ramified. Hence, image of the open Berkovich disk $B_{\vec{v}}$ at any direction $\vec{v}$ at $\zeta_{Gauss}$ does not contain $\zeta_{Gauss}$. Hence, by Theorem \ref{tool1}, every direction at $\zeta_{Gauss}$ is a good direction.

    $(3) \iff (4)$ This result is due to Baker and Rumely (\cite[Theorem~9.42]{brume}). 

    $(4) \implies (2)$ We prove it by contradiction. Suppose the Gauss point $\zeta_{Gauss}$ is not totally ramified. Let $\beta \neq \zeta_{Gauss}$ be another preimage of $\zeta_{Gauss}$. Consider the arc $[\zeta_{Gauss}, \beta]$. Image of this arc is a loop at $\zeta_{Gauss}$. As $\PP^{1,an}$ is a tree, there is a type~II point $\alpha$ in $(\zeta_{Gauss}, \beta)$ where the image of $\phi$ backtracks/retraces. In other words, under the map $\phi$ the two distinct directions at $\alpha$ containing $\zeta_{Gauss}$ and $\beta$ maps to the same direction. Thus the multiplicity of $\phi$ at $\alpha$ is strictly greater than the directional multiplicity at $\alpha$ in the direction containing $\zeta_{Gauss}$. So, the multiplicity function makes a jump at $\alpha$ on the arc $[\zeta_{Gauss}, \beta]$. We arrive at a contradiction.
\end{proof}

The following theorem is due to Faber (\cite[Theorem~8.2]{faber1}).

\begin{theorem}\label{tool5}
    Let $\phi$ be a degree $\geq 2$ rational map, in normalized form, defined over $K$. If $\phi$ has potential good reduction, then the locus of total ramification of $\phi$ is non-empty, closed, connected set. Moreover, the ramification locus of $\phi$ is also connected. 
\end{theorem}

The following theorem is due to Rivera-Letelier (\cite[Lemmas~5.3~and~5.4]{rl2}, \cite[Lemma~10.80]{brume}). In \cite[Lemma~10.80]{brume}, the field $K$ is assumed to have characteristic $0$, but the proof works perfectly fine for finite characteristic as well.

\begin{theorem}\label{tool7}
    Let $\phi$ be a rational function of degree $\geq 2$ defined over $K$. If $x$ is a fixed point of $\phi$ of type~III or IV, then, 

    $(1)$ $x$ is an indifferent fixed point.

    $(2)$ $(T_{x}\phi)(\vec{v}) = \vec{v}$, for every $\vec{v} \in T_{x}\PP^{1,an}$.
\end{theorem}

The following theorem is also due to Rivera-Letelier (\cite{rl3}, \cite[Theorem~10.85]{brume}), where the field $K$ is assumed to have characteristic zero but the proof extends directly for arbitrary characteristic.

\begin{theorem}\label{tool6}
    Let $\phi$ be a non-constant rational map defined over $K$. Let $X \subseteq \PP^{1,an}$ be a non-empty, connected, closed set under the weak topology such that $\phi(X) \subseteq X$. Then $X$ contains a point fixed by $\phi$. 
\end{theorem}

\section{Connected components of Berkovich fixed locus}

Let $K$ be a complete, algebraically closed, non-Archimedean field equipped with a non-trivial absolute value. Let $\PP^{1,an}$ be the Berkovich projective line over $K$. Let $\PP^{1}(K)$ be the set of rigid/classical/type~I points in $\PP^{1,an}$. 

Let $[\zeta, \xi]$ be an arc in $\PP^{1,an}$. We say a rational map $\phi$ \emph{backtracks} on the arc $[\zeta, \xi]$ if there is a point $\eta \in (\zeta, \xi)$ such that $T_{\eta}\phi$ maps the direction at $\eta$ containing $\zeta$ and the direction at $\eta$ containing $\xi$ to the same direction at $\phi(\eta)$. In such a scenario, it follows from Equation \eqref{2.2} that the multiplicity of $\phi$ at $\eta$ is strictly greater than multiplicity of points on $[\zeta, \xi]$ which are sufficiently close to $\eta$ under the big metric.

The following lemma is a particular case of Lemma $9.38$ of \cite[Page~277]{brume}.

\begin{lemma}\label{bijlemma}
    Let $\phi$ be a rational map of degree $\geq 2$ defined over $K$, with potential good reduction. Let $\zeta$ be a type~II totally ramified fixed point of $\phi$. For every point $\xi \in \PP^{1,an}$, $\phi$ maps the arc $[\zeta, \xi]$ bijectively onto the arc $[\zeta, \phi(\xi)]$.  
\end{lemma}

\begin{proof}
    As continuous image of a connected set is connected, $[\zeta, \phi(\xi)] \subseteq \phi([\zeta, \xi])$. As $\zeta$ is totally ramified and fixed, from Theorem \ref{tool4} there is no backtracking for $\phi$ on $[\zeta,\xi]$. A point $\eta$ of $\phi([\zeta, \xi])$ is called an \emph{endpoint} (\emph{leaf}) of $\phi([\zeta, \xi])$ if only one open Berkovich disk with boundary point $\eta$ intersects $\phi([\zeta, \xi])$. If $\phi([\zeta, \xi])$ has an endpoint other than $\zeta$ and $\phi(\xi)$, then $\phi$ must backtrack on the arc $[\zeta, \xi]$. As $[\zeta, \xi]$ is a closed arc, $\phi([\zeta, \xi])$ must be closed too. Hence, $\phi([\zeta, \xi])$ is a closed arc with end points $\zeta$ and $\phi(\xi)$. In other words, $\phi([\zeta, \xi])=[\zeta, \xi]$. Moreover, as $\phi$ does not backtrack on the arc $[\zeta, \xi]$, $\phi$ must be injective on the arc $[\zeta, \xi]$. The lemma is proved.
    \end{proof}

The following lemma is due to Rivera-Letelier \cite[Theorem~4]{rl4}, \cite[Theorem~10.105]{brume}. Nevertheless, we provide a proof of it here.

\begin{lemma}\label{3.5}
    Let $\phi$ be a rational map of degree $\geq 2$ defined over $K$. The map $\phi$ can have at most one type~II totally ramified fixed point. Moreover, $\phi$ has potential good reduction if and only if $\phi$ has a type~II totally ramified fixed point.
\end{lemma}

\begin{proof}
    Let us assume that $\phi$ has two totally ramified type~II fixed points $\zeta$ and $\xi$. From Lemma \ref{bijlemma}, $\phi$ maps the arc $[\zeta, \xi]$ bijectively onto itself. From Theorem \ref{tool5}, the arc $[\zeta, \xi]$ is contained in the locus of total ramification. As both $\zeta$ and $\xi$ are type~II points, the length of the arc $[\zeta, \xi]$ is finite under $\rho$. Using Theorem \ref{tool3}, we can cover the arc $[\zeta,\xi]$ with finitely many subarcs such that $\phi$ elongates the length of each of these subarcs by the degree of $\phi$. As degree of $\phi$ is $\geq 2$, we reach a contradiction.

    The last part of the theorem is a direct corollary of Theorem \ref{tool4}.
\end{proof}

\begin{lemma}\label{3.1}
    Let $\phi$ be a degree $\geq 2$ rational map, defined over $K$, with potential good reduction. Let $\zeta$ be the type~II totally ramified fixed point of $\phi$. Let $\vec{v}$ be a direction at $\zeta$ that is fixed under $T_{\zeta}\phi$. Let $B_{\vec{v}}$ be the open Berkovich disk with boundary point $\zeta$, in the direction of $\vec{v}$. Then, there is a point $\beta \in B_{\vec{v}}$ such that $\beta$ lies on the arc $[\zeta, \phi(\beta)]$.
\end{lemma}

\begin{proof}
    As $\vec{v}$ is a fixed direction under the reduction of $\phi$ at $\zeta$, by Theorems \ref{tool1} and \ref{tool4} we get $\phi(B_{\vec{v}}) = B_{\vec{v}}$. Moreover, by Theorem \ref{tool3}, we can find a point $\gamma$ sufficiently close to $\zeta$ such that $\rho (\zeta, \phi(\delta)) = m_{\phi}(\zeta, \vec{v}) \rho(\zeta, \delta)$, for every $\delta \in [\zeta, \gamma]$. If $\gamma \in [\zeta, \phi(\gamma)]$, then the point $\gamma$ is our required point $\beta$ in the lemma. Suppose that $\gamma$ does not lie in $[\zeta, \phi(\gamma)]$. Consider the unique point $\beta$ lying in the intersection of the three arcs $[\zeta, \gamma], [\zeta, \phi(\gamma)]$ and $[\gamma, \phi(\gamma)]$. We claim that $\beta$ satisfies the conditions of the lemma. As $\zeta$ is totally ramified and fixed under $\phi$ and the open Berkovich disk $B_{\gamma}(\beta)$ containing $\gamma$ with boundary point $\beta$ does not contain $\zeta$, the image $\phi(B_{\gamma}(\beta))$ is a open Berkovich disk containing $\phi(\gamma)$ with boundary point $\phi(\beta).$ Let us assume that $\phi(B_{\gamma}(\beta))$ contains $\beta$. As $\phi(B_{\gamma}(\beta))$ does not contain $\zeta$, the point $\phi(\beta)$ (which is the boundary point of $\phi(B_{\gamma}(\beta))$) must lie on the arc $[\zeta, \beta]$. But, this contradicts Theorem \ref{tool3}. So, $\phi(B_{\gamma}(\beta))$ does not contain $\beta$. As $\phi(B_{\gamma}(\beta))$ contains $\phi(\gamma)$, hence also the arc $(\phi(\beta), \phi(\gamma)]$, the point $\beta$ can not lie in $(\phi(\beta), \phi(\gamma)]$. As $\beta$ lies in $[\zeta, \phi(\gamma)]$, the point $\beta$ must lie in $[\zeta, \phi(\beta)]$.
\end{proof}

\begin{lemma}\label{3.8}
    Let $\phi$ be a rational map over $K$ of degree $\geq 2$ with potential good reduction. Let $\zeta$ be the type~II totally ramified fixed point of $\phi$. Let $\vec{v}$ be a residually fixed critical direction at $\zeta$, i. e. $\vec{v}$ is a fixed critical point under the map $T_{\zeta}\phi$. Let $\xi$ be a fixed point of $\phi$ lying in the open Berkovich disk $B_{\vec{v}}$. Then, $\xi$ must be of type~I, attracting and the unique fixed point of $\phi$ lying in $B_{\vec{v}}$. 
\end{lemma}

\begin{proof}
    Let us assume that $\xi$ is not of type~I. Consider the arc $[\zeta, \xi]$. From Lemma \ref{bijlemma}, $\phi([\zeta, \xi]) = [\zeta, \xi]$ and $\phi$ is one-one on $[\zeta, \xi]$. As $\xi$ is not of type~I, the arc $[\zeta, \xi]$ is of finite length under the big metric $\rho$ (see Equation \eqref{bigmet}). Hence, by Theorem \ref{tool3}, the multiplicity of $\phi$ at every point on $(\zeta, \xi)$ must be $1$. From Theorem \ref{tool2}, the direction $\vec{v}$ is not a residual critical direction of $\phi$. We arrive at a contradiction.

    Next we need to show that this classical fixed point must be attracting. Consider the arc $[\zeta, \xi]$. This arc maps to itself bijectively under $\phi$. As $\vec{v}$ is a critical point of $T_{\zeta}\phi$, from Theorem \ref{tool2} the directional multiplicity of $\phi$ along the direction $\vec{v}$ is strictly greater than $1$. From Theorem \ref{tool3}, there is a point $\beta$ lying on $[\zeta, \xi]$, such that $\phi(\beta)$ lies on the arc $(\beta, \xi)$. As $\phi$ is one-one on $[\zeta, \xi]$, for every point $z$ lying on the arc $[\beta, \xi)$ the point $\phi(z)$ lies on the arc $(z, \xi)$. Hence, $\xi$ is an attracting fixed point.

    Lastly we need to show uniqueness of $\xi$. Let $\eta$ be another fixed point of $\phi$ lying in $B_{\vec{v}}$. Then it must be of type~I and attracting. Consider the arc $[\zeta, \eta]$. By the same argument as in the previous paragraph, every point in $(\zeta, \eta)$ is attracted towards $\eta$ under $\phi$. As $\xi$ and $\eta$ lie in the same direction at $\zeta$, the arcs $[\zeta, \xi]$ and $[\zeta, \eta]$ share an initial segment. Let the intersection $[\zeta, \xi] \cap [\zeta, \eta]$ be $[\zeta, \alpha]$. Then, $\phi(\alpha)$ lies on the intersection $(\alpha, \xi] \cap (\alpha, \eta]$ which is empty. We arrive at a contradiction.  
\end{proof}

The following lemma is well-known (See \cite[Lemma~2.1]{bene1}), but we give a proof of it here.

\begin{lemma}\label{norep}
    Let $\phi$ be a rational map of degree $\geq 2$ with potential good reduction. There is no classical repelling fixed point of $\phi$.
\end{lemma}

\begin{proof}
    Without loss of generality, we can assume that the map $\phi$ has good reduction. Let $z$ be a classical repelling fixed point of $\phi$. Consider the arc $[z, \zeta_{Gauss}]$. From Lemma \ref{bijlemma}, the rational map $\phi$ maps the arc $[z, \zeta_{Gauss}]$ bijectively onto itself. As $z$ is repelling, there exists a point $\omega$ lying on the arc $(z, \zeta_{Gauss}]$ such that $\phi(\omega)$ lies on the arc $(\omega, \zeta_{Gauss}]$. So $\phi$ maps the arc $[\omega, \zeta_{Gauss}]$ bijectively onto its subarc $[\phi(\omega), \zeta_{Gauss}]$. But this contradicts Theorem \ref{tool3}. Hence, the lemma is proved. 
\end{proof}

\begin{theorem}\label{maintheorem}
    Let $\phi$ be a rational map of degree $\geq 2$ with potential good reduction. Then, the number of connected components of the fixed locus of $\phi$ is $1 \; +$ the number of attracting fixed points of $\phi$.
\end{theorem}

\begin{proof}
    Without loss of generality, we can assume that the map $\phi$ has good reduction. From Lemma \ref{norep}, we know that $\phi$ does not have any classical repelling fixed points. To prove the theorem, it is enough to show that every fixed point in the hyperbolic space and every classical indifferent fixed point lie in the same connected component as the Gauss point. 

    Let $\xi$ be a hyperbolic fixed point of $\phi$. Consider the arc $[\zeta_{Gauss}, \xi]$. As $\zeta_{Gauss}$ is totally ramified and fixed under $\phi$, from Lemma \ref{bijlemma} the rational map $\phi$ maps the arc $[\zeta_{Gauss}, \xi]$ bijectively onto itself. As the arc $[\zeta_{Gauss}, \xi]$ is of finite length under the big metric $\rho$, from Theorem \ref{tool3} the multiplicity of any point lying in $(\zeta_{Gauss}, \xi)$ under $\phi$ must be $1$. So the entire arc $[\zeta_{Gauss}, \xi]$ must consist of fixed points.

    Let $z$ be a classical indifferent fixed point. Consider the arc $[\zeta_{Gauss}, z]$. From Lemma \ref{bijlemma}, we get that $\phi$ maps $[\zeta_{Gauss}, z]$ bijectively onto itself. From Lemma \ref{3.8}, we get that the directional multiplicity of $\phi$ at $\zeta_{Gauss}$ along the direction containing $z$ must be $1$. From Theorem \ref{tool4}, every point on the arc $[\zeta_{Gauss}, z]$ is of multiplicity $1$. Let $\eta$ be a point in $[\zeta_{Gauss}, z)$. The length of the arc $[\zeta_{Gauss}, \eta]$ under the big metric is finite. As every point on this arc has multiplicity $1$, we can cover the arc with finitely many subarcs such that Theorem \ref{tool3} holds on each of there arc. Thus, $\rho(\zeta_{Gauss}, \phi(\eta))$ must be equal to $\rho(\zeta_{Gauss}, \eta)$. As $\phi$ maps the arc $[\zeta_{Gauss}, z]$ to itself with no backtracking (Lemma \ref{bijlemma}), $\eta$ must be a fixed point of $\phi$. Thus, the entire arc $[\zeta_{Gauss}, z]$ consists of fixed points.    
\end{proof}

\begin{remark}\label{rem:unique_rep_per_pt}
    Observe that from the proof of Theorem \ref{maintheorem} and Lemma \ref{norep}, we get that a rational map with potential good reduction has a unique repelling fixed point. The map $\phi$ having potential good reduction is equivalent to $\phi^{n}$ having potential good reduction for every $n \in \N$. Thus, one observes that a map with potential good reduction has a unique repelling periodic point (which is necessarily a fixed point) in $\PP^{1,an}$. This is a well-known result due to Rivera-Letelier (See \cite[Theorem~2]{rl4}).  
\end{remark}

\begin{remark}\label{rem:bene_new_proof}
    Let $\phi$ be a rational map with potential good reduction with $\zeta$ being the unique repelling fixed point of $\phi$. Let $a$ be a classical indifferent fixed point of $\phi$. From Theorem \ref{maintheorem}, the direction $\vec{v}$ at $\zeta$ towards $a$ is a fixed point of $T_{\zeta}\phi$ which is not critical. Thus, $\vec{v}$ is a good direction with directional multiplicity $1$ under $\phi$. Let $B_{\vec{v}}$ be the open Berkovich disk with boundary point $\zeta$ along the direction $\vec{v}$. As $\vec{v}$ is a good direction with directional multiplicity $1$, $\phi$ maps $B_{\vec{v}}$ to itself bijectively. Thus, any point $b$ in the set $\phi^{-1}(a) \setminus a$ lies along a different direction $\vec{w}$ at $\zeta$. As $\vec{w}$ is a good direction at $\zeta$ which is not fixed under $T_{\zeta}\phi$, any point $c \in \phi^{-1}(b)$ lies along another direction $\vec{u}$ at $\zeta$. Thus, $\zeta$ is the unique point lying on the arcs joining any two points among $\{a,b,c\}$. On the other hand, from Lemma \ref{3.8} any three distinct attracting fixed points of $\phi$ lie in three distinct directions at $\zeta$. Thereby, we obtain another proof of the main theorem of \cite{bene1}.
\end{remark}

\begin{corollary}\label{3.10}
    Let $\phi$ be a rational map of degree $d \geq 2$ with potential good reduction. Then, the number of connected components of fixed locus of $\phi$ is bounded above by $(d+2)$.
\end{corollary}

\begin{proof}
    The corollary follows directly from Theorem \ref{maintheorem}.
\end{proof}

Recall that, the \emph{hyperbolic fixed locus} of a rational map $\phi$ is said to be the intersection of the fixed locus with the Berkovich hyperbolic space. In other words, the hyperbolic fixed locus of $\phi$ is the set of all non-rigid (non-classical) fixed points of $\phi$. From Corollary \ref{3.10}, the following corollary follows directly.

\begin{corollary}\label{3.11new}
    Let $\phi$ be a degree $\geq 2$ rational map, with potential good reduction. Then the hyperbolic fixed locus of $\phi$ is connected. \qed
\end{corollary}

We now study the extreme cases of Theorem \ref{maintheorem}.

\begin{corollary}\label{3.2}
    Let $\phi$ be a degree $\geq 2$ rational map defined over $K$, with potential good reduction. Then the following are equivalent,
    
    $(1)$ the fixed locus of $\phi$ is connected,
    
    $(2)$ the rational map $\phi$ does not have any attracting fixed point,

    $(3)$ the reduction of $\phi$ at the hyperbolic totally ramified fixed point has no fixed critical point. 
\end{corollary}

\begin{proof} 
    Without loss of generality, we can assume that the rational map $\phi$ has good reduction.

    $(1) \iff (2)$ From Theorem \ref{maintheorem}, this is direct.

    $(1) \implies (3)$ Let us assume that the fixed locus of $\phi$ is connected. We want to show that the reduction of $\phi$ has no fixed critical point. We will prove by contradiction. Suppose the reduction of $\phi$, $T_{\zeta_{Gauss}}\phi$, has a fixed critical point corresponding to the direction $\vec{v}$ at $\zeta_{Gauss}$. Let $B_{\vec{v}}$ be the open Berkovich disk at $\zeta_{Gauss}$ in the direction of $\vec{v}$. As the direction $\vec{v}$ is fixed under the reduction of $\phi$ and $\phi$ has good reduction, by Theorems \ref{tool1} and \ref{tool4} we get $\phi(B_{\vec{v}}) = B_{\vec{v}}$. Moreover, the direction $\vec{v}$ is a critical point for $T_{\zeta_{Gauss}}\phi$. According to Theorem \ref{tool2}, part $(3)$, this is equivalent to the fact that the directional multiplicity of $\phi$ at $\zeta_{Gauss}$ in the direction of $\vec{v}$ is $2$ or higher. Hence, by Lemma \ref{3.1} and Theorem \ref{tool3} we can choose a type~II point $\beta \in B_{\vec{v}}$ sufficiently close to $\zeta_{Gauss}$ such that $\phi(\beta) \neq \beta$ and $\beta$ lies on the arc $[\zeta_{Gauss}, \phi(\beta)]$. Let $\vec{w}$ be the direction at $\beta$ containing $\phi(\beta)$, and $B_{\vec{w}}$ is the corresponding open Berkovich disk with boundary point $\beta$. We claim that $\phi(B_{\vec{w}}) \subset B_{\vec{w}}$. To prove the claim, first observe that as $\phi$ has good reduction and $\zeta_{Gauss} \notin B_{\vec{w}}$, $\phi(B_{\vec{w}})$ can not contain $\zeta_{Gauss}$. As $\beta$ and $\zeta_{Gauss}$ lie in the same direction at $\phi(\beta)$, $\phi(B_{\vec{w}})$ is an open Berkovich disk at $\phi(\beta)$ which does not contain $\beta$ or $\zeta_{Gauss}$. Hence, $\phi(B_{\vec{w}}) \subset B_{\vec{w}}$. Consider the non-empty closed connected set $X = \bar{B}_{\vec{w}}$, the weak closure of $B_{\vec{w}}$. As $\phi(X) \subset X$, by Theorem \ref{tool6} there exists a fixed point of $\phi$ in $\bar{B}_{\vec{w}}$. Let us denote this point as $\eta$. The arc $[\zeta_{Gauss}, \eta]$ is contained in $B_{\vec{v}}$. As directional multiplicity of $\phi$ at $\zeta_{Gauss}$ in the direction of $\vec{v}$ is $2$ or higher, the arc $[\zeta_{Gauss}, \eta]$ contains points, such as $\beta$, that are not fixed under $\phi$. Hence, $\zeta_{Gauss}$ and $\eta$ lie in distinct connected components of the fixed locus. We arrive at a contradiction.

    $(3) \implies (1)$ For the converse, as $\phi$ has good reduction, the Gauss point $\zeta_{Gauss}$ is a totally ramified fixed point of $\phi$. Let $\xi$ be another fixed point of $\phi$. To prove connectedness of the fixed locus, it is enough to show that every point on the arc $[\zeta_{Gauss}, \xi]$ is fixed under $\phi$. Let $B_{\xi}(\zeta_{Gauss})$ be the open Berkovich disk containing $\xi$ with boundary point $\zeta_{Gauss}$. As $\xi$ is fixed, the set $\phi(B_{\xi}(\zeta_{Gauss}))$ contains $B_{\xi}(\zeta_{Gauss})$. As the Gauss point is totally ramified and fixed under $\phi$, $\phi(B_{\xi}(\zeta_{Gauss})) = B_{\xi}(\zeta_{Gauss})$. So, the direction at $\zeta_{Gauss}$ corresponding to $B_{\xi}(\zeta_{Gauss})$ is fixed under the reduction of $\phi$. By assumption of the corollary, this direction is not a critical point for the reduction of $\phi$. So the directional multiplicity of $\phi$ at $\zeta_{Gauss}$ in the direction of $B_{\xi}(\zeta_{Gauss})$ is $1$. From Theorem \ref{tool4}, along every direction at $\zeta_{Gauss}$ the multiplicity function of points is non-increasing. Hence, the multiplicity of every point in $B_{\xi}(\zeta_{Gauss})$ is $1$. As $\zeta_{Gauss}$ and $\xi$ are both fixed points and every point on the arc $[\zeta_{Gauss}, \xi]$ has multiplicity $1$, an argument similar to the one in Theorem \ref{maintheorem} shows that the entire arc $[\zeta_{Gauss}, \xi)$ consists of fixed points of $\phi$. Hence, $[\zeta_{Gauss}, \xi]$ lies in the fixed locus of $\phi$.
\end{proof}

Next, we give a few equivalent criterion for the fixed locus to be a finite set.

\begin{corollary}\label{3.11}
    Let $\phi$ be a degree $d \geq 2$ rational map with potentially good reduction. Let $\zeta$ be the unique type~II totally ramified fixed point of $\phi$. Then, the following are equivalent,
    
    $1)$ Every fixed point of $T_{\zeta}\phi$ is a critical point.

    $2)$ Every classical fixed point of $\phi$ is attracting.

    $3)$ The fixed locus of $\phi$ consists of the classical fixed points of $\phi$ and $\zeta$.
    
    $4)$ The fixed locus of $\phi$ is a finite set.
\end{corollary}

\begin{proof}
    Without loss of generality, we can assume that $\phi$ has good reduction, i.e. $\zeta = \zeta_{Gauss}$. 
    
    $(1) \implies (2)$ As every fixed point of $T_{\zeta_{Gauss}}\phi$ is critical, by Theorem \ref{tool2}, every fixed direction of $\phi$ at the Gauss point has directional multiplicity $\geq 2$. By Theorem \ref{tool3}, the Gauss point is an isolated point of the fixed locus of $\phi$. By Theorem \ref{maintheorem}, every classical fixed point of $\phi$ must be attracting.

    $(2) \implies (3)$ Follows directly from Theorem \ref{maintheorem}.

    $(3) \implies (4)$ This is direct.

    $(4) \implies (1)$ We prove this by contradiction. Suppose that a fixed point $a$ of $T_{\zeta_{Gauss}}\phi$ is not critical. Then, the direction at the Gauss point corresponding to $a$ has directional multiplicity $1$ under $\phi$. From Lemma \ref{3.1}, there is an arc $[\zeta_{Gauss}, \beta]$ along the direction represented by $a$ such that every point on this arc is a fixed point of $\phi$. Thus, the fixed locus of $\phi$ consists of infinitely many points.
\end{proof}

Next, we give a family of examples which shows that the bound mentioned in Corollary \ref{3.10} is \emph{sharp}.

\begin{example}\label{3.12}
    Let $d \geq 2$ be a natural number. Consider the Berkovich projective line over an algebraically closed field $K$, complete with respect to a non-trivial non-Archimedean absolute value, such that the residue characteristic of $K$ divides $d$. Consider the map, $\phi(z) = z^{d}$. Every direction at $\zeta_{Gauss}$ is totally ramified under $\phi$. Hence, $\zeta_{Gauss}$ is an isolated fixed point of $\phi$. The other fixed points of $\phi$ must be of type~I. They consists of $\infty$ and the roots of $z^{d}-z$ in $K$. Hence, the fixed locus of $\phi$ consists of $(d+2)$ many distinct isolated fixed points of $\phi$.
\end{example}    

\begin{example}\label{ex:sep_fin}
    While in Example \ref{3.12}, the maps had inseparable reduction at the Gauss point, this is not necessary for attaining the bound. Here we will produce a rational function with good reduction that satisfies the equivalent conditions of Corollary \ref{3.11}, without having inseparable reduction at the Gauss point. Consider the map,
    \begin{equation*}
        \phi(z) = \frac{z^{3} + 2z^{2}}{z + 1},
    \end{equation*}
    defined over a field $K$ with residue characteristics $3$. One directly observes that the map $\phi$ has good reduction and $T_{\zeta_{Gauss}}\phi$ has the same expression as $\phi$. The numerator and denominator of $T_{\zeta_{Gauss}}\phi$ are coprime to eah other. The point $\infty$ is a superattracting fixed point of $T_{\zeta_{Gauss}}\phi$. The other fixed points of $T_{\zeta_{Gauss}}\phi$ are the solutions of $(z^{3} + 2z^{2}) - z (z+1) = z(z^{2} + z -1)$. On the other hand, the critical points of $T_{\zeta_{Gauss}}\phi$ other than $\infty$ are the solutions of $(z+1)(3z^{2} + 4z) - (z^{3} + 2z^{2}) = -z (z^{2} + z - 1)$. Thus, every fixed point of $T_{\zeta_{Gauss}}\phi$ is a critical point. From Theorem \ref{3.11}, the fixed locus of $\phi$ is a finite set consisting of $\zeta_{Gauss}, 0, \infty$ and the solutions of $z^{2} +z -1$.  
\end{example}

Next, we characterize the fixed locus of a rational map with potential good reduction. We call a M\"{o}bius map, defined over an algebraically closed field $K$, a \emph{translation map} if it has a unique fixed point in the classical projective line over $K$. Let $A_{\phi}$ denote the finite set of points of $\PP^{1,an}$ consisting of all classical indifferent fixed points and the repelling fixed points of $\phi$. Let $\text{Hull}(A_{\phi})$ be the \emph{connected hull} of the points lying in $A_{\phi}$, i.e. the smallest closed connected subset of $\PP^{1,an}$ containing the points of $A_{\phi}$ (compare \cite[Page~441]{faber1}). Following Rumely's notion in \cite[Page~849]{rume1}, a type~II indifferent fixed point of $\phi$ is called \emph{id-indifferent} if the reduction of $\phi$ at that point is the identity map. A point $\zeta$ lying on a connected subset $S$ of $\PP^{1,an}$ is called a \emph{branch point} if $S \backslash \{\zeta\}$ has more than $2$ connected components.

\begin{corollary}\label{3.6}
    Let $\phi$ be a rational map with potential good reduction. The fixed locus of $\phi$, denoted $\text{Fix}(\phi)$, contains $\text{Hull}(A_{\phi})$. Apart from id-indifferent points, the only possible branch point of $\text{Fix}(\phi)$ is the totally ramified fixed point of $\phi$. 
\end{corollary}

\begin{proof}
    From the proof of Theorem \ref{maintheorem}, we get that the $\text{Hull}(A_{\phi})$ is contained in the fixed locus of $\phi$.
    
    Let $\zeta$ be a branch point of $\text{Fix}(\phi)$ such that $\zeta$ is not id-indifferent. Then, there are at least $3$ directions at $\zeta$ which are fixed under the reduction of $\phi$ at $\zeta$. As $\zeta$ is not id-indifferent, the reduction of $\phi$ at $\zeta$ must have degree $\geq 2$. By Theorem \ref{tool2}, $\zeta$ must be a repelling type~II fixed point of $\phi$. By Remark \ref{rem:unique_rep_per_pt}, $\zeta$ must be the unique totally ramified fixed point of $\phi$.   
\end{proof}

The following lemma is a particular case of the \emph{first identification lemma}, due to Rumely \cite[Lemma~2.1]{rume1}.

\begin{lemma}\label{tool8}
    Let $\phi$ be a rational map of degree $\geq 2$ with potential good reduction. Let $\zeta \in \PP^{1,an}$ be the type~II totally ramified, fixed point of $\phi$. Let $\vec{v}$ be a direction at $\zeta$, and $B_{\vec{v}}$ the corresponding open Berkovich disk with boundary point $\zeta$. Then, fixed point multiplicity of $\vec{v}$ under the reduction of $\phi$ at $\zeta$ is the same as the number of classical fixed points of $\phi$ lying in $B_{\vec{v}}$, counted with fixed point multiplicity.
\end{lemma}

\begin{proof}
    As $\zeta$ is totally ramified and fixed, from Theorem \ref{tool4} every direction at $\zeta$ is a good direction under $\phi$. Thus, the surplus multiplicity of any direction at $\zeta$ under $\phi$ is zero. The lemma follows as a particular case of the first identification lemma of Rumely (\cite[Lemma~2.1]{rume1}).
\end{proof}

\begin{proposition}\label{noid}
    Let $\phi$ be a rational map of degree $\geq 2$ with potential good reduction. If $\text{Hull }(A_{\phi})$ does not contain any id-indifferent point, then $\text{Fix}(\phi) = \text{Hull }(A_{\phi}) \cup \{\text{the set of attracting fixed points of } \phi \}$. 
\end{proposition}

\begin{proof}
    Without loss of generality, we can assume that the map $\phi$ has good reduction. Let $\xi$ be a non-attracting fixed point of $\phi$ that does not lie in $\text{Hull}(A_{\phi})$. Thanks to Lemma \ref{norep}, any non-attracting classical fixed point of $\phi$ must be indifferent. Thus, $\xi$ can not be a classical point. Also, by our assumption, $\xi$ is distinct from the Gauss point. From the proof of Theorem \ref{maintheorem}, $\xi$ lies in the connected component of $\text{Fix}(\phi)$ containing the Gauss point. Thus, the arc $[\zeta_{Gauss}, \xi]$ consists of fixed points of $\phi$. So the direction $\vec{v}$ at the Gauss point towards $\xi$ is a fixed point of $T_{\zeta_{Gauss}}\phi$, which is not a critical point of $T_{\zeta_{Gauss}}\phi$. From Lemma \ref{tool8}, there must be at least one classical fixed point lying in the open Berkovich disk $B_{\vec{v}}$, with boundary point $\zeta_{Gauss}$. Let $a$ be such a classical fixed point of $\phi$ lying in $B_{\vec{v}}$. From Lemma \ref{3.8}, $a$ must be an indifferent fixed point. As $\xi$ is not contained in Hull$(A_{\phi})$, $\xi$ does not lie on the arc $[\zeta_{Gauss}, a]$. Let $\eta$ be the unique type~II point lying on the three arcs $[\zeta_{Gauss}, \xi]$, $[\zeta_{Gauss}, a]$ and $[a, \xi]$. As $\eta$ lies in $[\zeta_{Gauss}, \xi]$, it must be a fixed point of $\phi$. Let $B_{a}$ be the open Berkovich disk with boundary point $\eta$, containing $a$. As $\phi(B_{a})$ contains $a$ but does not contain the Gauss point, $\phi(B_{a}) = B_{a}$. Thus, all the three distinct directions at $\eta$ containing $\zeta_{Gauss}, \xi$ and $a$ are fixed directions under the reduction of $\phi$ at $\eta$. As $\eta$ lies on the arc $[\zeta_{Gauss}, a] \subset Hull(A_{\phi})$, by assumption of the corollary $\eta$ is not id-indifferent. Thus, $\eta$ must be a repelling fixed point of $\phi$. As $a$ and $\xi$ lie along the same direction at $\zeta_{Gauss}$, $\eta$ is distinct from $\zeta_{Gauss}$. But the Gauss point is the unique repelling fixed point of $\phi$ (See Remark \ref{rem:unique_rep_per_pt}). We arrive at a contradiction.
\end{proof}

\begin{example}\label{3.3}
    Suppose that $K$ has residue characteristic $\neq 2$. Consider the rational function,
    \begin{equation}
        \phi(z) = \frac{z^{2}-z}{z-2}.
    \end{equation}
    The map $\phi$ has good reduction and it has the same expression as its residual map $T_{\zeta_{Gauss}\phi}$. The residual map has fixed points at $\infty$ and $0$, none of which are critical points of the residual map. By Corollary \ref{3.2}, $\phi$ has no attracting fixed point and its fixed locus must be connected. By Corollary \ref{3.6}, $\text{Fix}(\phi)$ contains the arc $[0, \infty]$ which is the connected hull of $A_{\phi}$. 
    
    Let $a \in K^{\times}$ with $|a| < 1$. Normalization of the conjugation of $\phi$ by the M\"{o}bius map $z \mapsto az$ results in the map $z \mapsto \frac{az^{2} - z}{az - 2}$. Thus, reduction of the $\phi$ at any point on the arc $(\zeta_{Gauss}, 0)$ is of the form $z \mapsto \frac{z}{2}$. Thus, no point on the arc $[\zeta_{Gauss}, 0]$ is id-indifferent for $\phi$. So the intersection of $\text{Fix}(\phi)$ with the open Berkovich disk with boundary point $\zeta_{Gauss}$ and containing $0$ is the arc $(\zeta_{Gauss}, 0]$.

    Let $a \in K^{\times}$ with $|a| > 1$. Normalization of the conjugation of $\phi$ by the M\"{o}bius map $z \mapsto az$ results in the map $z \mapsto \frac{z^{2} - z/a}{z - 2/a}$. As the residue characteristic of $K$ is $\neq 2$, the reduction of $\phi$ at any point on the arc $[\zeta_{Gauss}, \infty]$ is the identity map. In such a case, we do not know how to determine the exact fixed locus of $\phi$ near id-indifferent points.
\end{example}

\begin{example}\label{3.7}
    We give an example of a rational map where we can determine the entire fixed locus of it. Consider the rational map over $K$,

    \begin{equation}
        \phi(z) = \frac{2z^{2} + z}{z + 2}.
    \end{equation}
    The numerator and denominator of $\phi$ share a root over fields of characteristic $2$ or $3$. So $\phi$ has good reduction as long as residue characteristic is $\neq 2,3$. Moreover, the classical fixed points of $\phi$ are $0,1$ and $\infty$. The multiplier at $0$ and $\infty$ are $1/2$ and the multiplier at $1$ is $4/3$. Thus, if $K$ has residue charateristic $2$ or $3$, then $\phi$ has a repelling fixed point. By Lemma \ref{norep}, $\phi$ does not have potential good reduction over fields of residue characteristic $2$ or $3$. So, we study the fixed locus of $\phi$ under the assumption that residue characteristics of $K$ is $\neq 2,3$. Over such a field $K$, $\phi$ has no attracting fixed point. By Theorem \ref{maintheorem}, the fixed locus of $\phi$ is connected. By Corollary \ref{3.6}, $\text{Fix}(\phi)$ contains the connected hull of $\{0,1,\infty\}$. A similar computation as in the previous example shows that there are no id-indifferent fixed points of $\phi$ lying on the connected hull of $\{0,1,\infty\}$. Thus, the fixed locus of $\phi$ is the connected hull of ${0,1,\infty}$ in $\PP^{1,an}$.
\end{example}

Let $n$ be a natural number. Parallel to the fixed locus, one defines the \emph{Berkovich $n$-periodic locus} of a rational map $\phi$ to be the set of all $n$-periodic points of $\phi$. The set of all non-classical points lying on the Berkovich $n$-periodic locus of a raional map $\phi$ will be called the \emph{hyperbolic $n$-periodic locus} of $\phi$. Let $\zeta \in \PP^{1,an}$ be a periodic point of $\phi$. The \emph{exact period of $\zeta$} is the smallest natural number $n$ such that $\phi^{n}(\zeta) = \zeta$. A classical periodic point $a$ of exact period $n$ is said to be \emph{attracting} if $|(\phi^{n})'(\zeta)| < 1$.

A rational map $\phi$ defined over $K$ has potential good reduction if and only if the $n$th self-iterate of $\phi$, denoted $\phi^{n}$, has potential good reduction, for every $n \in \N$. Thus from the results above, we get the following direct corollary.

\begin{corollary}\label{cor:berk_per_loc}
    Let $\phi$ be a degree $\geq 2$ rational map with potential good reduction. Then, 

    $1)$ The number of connected components of the Berkovich $n$-periodic locus is $1+$ the number of attracting $n$-periodic points.
    
    $2)$ The hyperbolic $n$-periodic locus of $\phi$ is connected, for every $n \in \N$.

    $3)$ The Berkovich $n$-periodic locus of $\phi$ has at most $(d^{n}+2)$ many components. 
    
    $4)$ The Berkovich $n$-periodic locus of $\phi$ is finite if and only if every fixed point of $T_{\zeta}\phi^{n}$ is a critical point, where $\zeta$ is the unique repelling fixed point of $\phi$. \qed
\end{corollary}

\end{document}